\pdfoutput=1

\documentclass[12pt,a4paper]{article}

\usepackage[english]{babel}
\usepackage[numbers,sort&compress]{natbib}
\usepackage[
  a4paper,
  top=2cm,
  bottom=2cm,
  left=2.6cm,
  right=2.8cm,
  marginparwidth=1.75cm
]{geometry}

\usepackage[utf8]{inputenc}
\usepackage[T1]{fontenc}

\usepackage{amsmath,amssymb,amsfonts,amsthm,mathtools}
\usepackage{bm}
\usepackage{bbm}
\usepackage{esvect}
\allowdisplaybreaks

\usepackage{graphicx}
\usepackage{booktabs}
\usepackage{adjustbox}
\usepackage{caption}
\usepackage{subcaption}
\usepackage{float}
\usepackage{placeins}

\usepackage{algorithm}
\usepackage{algpseudocode}

\usepackage{tikz}
\usetikzlibrary{decorations.pathreplacing,decorations.markings}
\tikzset{
  vtx/.style={circle, fill=black, draw=black, inner sep=2pt},
  gluevtx/.style={circle, fill=red, draw=red, inner sep=2.15pt},
  ed/.style={blue, line width=1.2pt}
}

\usepackage{pifont}
\usepackage{enumitem}
\usepackage{verbatim}
\usepackage{authblk}

\usepackage{xcolor}
\usepackage{titlesec}

\definecolor{mycitecolor}{RGB}{70,10,200}
\usepackage[
  colorlinks=true,
  linkcolor=blue,
  citecolor=mycitecolor,
  urlcolor=black
]{hyperref}

\usepackage[nameinlink,noabbrev]{cleveref}
\newcommand{\one}{\mathbf{1}}

\theoremstyle{plain}
\newtheorem{theorem}{Theorem}[section]
\newtheorem{lemma}[theorem]{Lemma}

\newtheorem{corollary}[theorem]{Corollary}
\newtheorem{conjecture}[theorem]{Conjecture}

\theoremstyle{definition}

\newtheorem{example}[theorem]{Example}

\theoremstyle{remark}

\title{\bf Local Tur\'an inequalities for walks and the spectral radius}

\author[1]{\bf Feng Liu\footnote{Email: liufeng0609@126.com.}}
\author[1]{\bf Shuang Sun\footnote{Email: chocolatesun@sjtu.edu.cn.}}
\author[1]{\bf Yan Wang\footnote{Email: yan.w@sjtu.edu.cn (corresponding author).}}
\author[2]{\bf Qi Wu\footnote{Email: wuqimath@163.com.}}
\affil[1]{\small School of Mathematical Sciences, Shanghai Jiao Tong University,  Shanghai, 200240, China}
\affil[2]{\small School of Mathematical Sciences, East China Normal University, Shanghai, 200241, China}

\date{}

\begin{document}

\maketitle

\begin{abstract}
Nikiforov's well-known spectral Tur\'an inequality for walks states that, for every graph $G$ with clique number $\omega(G)$, $\lambda^r(G)\le w_r(G)(1-1/\omega(G))$, where $\lambda(G)$ is the largest eigenvalue of the adjacency matrix of $G$, and $w_r(G)$ is the number of walks with $r$ vertices in $G$. For $r=1$, this is Wilf's inequality; for $r=2$, it gives Nikiforov's spectral Tur\'an theorem. Recently, Liu and Ning proved local versions of these two inequalities, strengthening both Wilf's inequality and Nikiforov's spectral Tur\'an theorem. It is natural to ask whether Nikiforov's spectral Tur\'an inequality for walks also admits a local strengthening. Motivated by this question, Kannan, Kumar, and Pragada conjectured the vertex-local bound $\lambda^r(G)\le \sum_{v\in V(G)} w_r(v)(1-1/c_G(v))$, where $w_r(v)$ denotes the number of walks with $r$ vertices starting at $v$, and $c_G(v)$ is the maximum order of a clique containing $v$. This conjecture is important because it gives the most natural local form of Nikiforov's spectral Tur\'an inequality for walks. In this paper, we confirm this conjecture. More precisely, for $r\ge 2$, we prove the stronger edge-local inequality
$$\lambda^r(G)
   \le
   \sum_{uv\in E(G)}
   \frac{c_G(uv)-1}{c_G(uv)}
   \bigl(w_{r-1}(u)+w_{r-1}(v)\bigr),$$
where $c_G(uv)$ is the maximum order of a clique containing the edge $uv$. Our result implies Nikiforov's spectral Tur\'an inequality for walks and unifies several local spectral extremal results of Liu and Ning. We also determine all extremal graphs for both the edge-local and vertex-local inequalities. The main new ingredient is a Markov-chain estimate whose transition matrix is constructed from a Perron vector of $A(G)$; this estimate carries the local edge coefficient through walks of arbitrary length.

\smallskip


\smallskip

\end{abstract}
\section{Introduction}
Throughout the paper, all graphs are finite and simple. For a graph $G$, we write $n=|V(G)|$ and $m=|E(G)|$. We denote by $A(G)$ the adjacency matrix of $G$ and by $\lambda(G)$ the largest eigenvalue of $A(G)$. We write $\lambda^r(G)$ for $(\lambda(G))^r$. 

Extremal graph theory studies how dense a graph can be when certain subgraphs are forbidden.  A classical result in this area is Tur\'an's
theorem~\cite{Turan41}, which gives the maximum number of edges in a graph with bounded clique number.  Later, Motzkin and Straus~\cite{MS65} gave an analytic proof of Tur\'an's theorem, connecting the clique number with matrix methods
and optimization. Spectral extremal graph theory studies similar problems from a spectral point
of view.  Instead of  counting edges, it estimates eigenvalues of matrices associated with a graph, especially the largest eigenvalue of the adjacency matrix.  This eigenvalue is closely related to the density of the graph: it depends not only on the number of edges, but also on how these edges are
distributed.  Thus many classical extremal results have spectral analogues, in which bounds on the number of edges are replaced by bounds on the largest eigenvalue. Nosal~\cite{Nosal70} proved that every triangle-free graph $G$ satisfies
$\lambda(G)\le \sqrt{m}$.  Wilf~\cite{Wilf86} proved that every graph $G$
satisfies $\lambda(G)\le n(1-1/\omega(G))$.  Nikiforov~\cite{Nikiforov02} later proved the stronger inequality
$\lambda(G)\le \sqrt{2m(1-1/\omega(G))}$, confirming a conjecture of Edwards
and Elphick~\cite{EdwardsElphick83}. In 2006, Nikiforov~\cite{N06} further studied the relationship between the largest eigenvalue and the number of walks in a graph.
 For further results on
spectral radius bounds and background on graph spectra, we refer the reader to
\cite{Stanley87,Hong88,HongShuFang01,BN07,BermanPlemmons94,BrouwerHaemers12,
CDS80,Seneta06,Nikiforov11Survey}. Let $w_r(G)$ denote the number of walks on $r$ vertices in $G$. Nikiforov's spectral Tur\'an inequality for walks states the following.  

\begin{theorem}[Nikiforov~\cite{N06}]\label{thm:Nikiforov-walk}
Let $G$ be a graph with at least one edge.  For every integer $r\ge 1$, 
\begin{flalign*}
\lambda^r(G) \le w_r(G)\frac{\omega(G)-1}{\omega(G)}.
\end{flalign*}
Moreover, equality holds if and only if one of the following two alternatives holds:
either $r=1$ and $G$ itself is a regular complete $\omega(G)$-partite graph; or
$r\ge 2$ and, after deleting all isolated vertices of $G$, the remaining graph is
either a regular complete $\omega(G)$-partite graph with $\omega(G)\ge 3$, or a
complete bipartite graph when $\omega(G)=2$; in the latter case, the complete
bipartite graph is regular if $r$ is odd.
\end{theorem}

This theorem unifies two important spectral inequalities. When $r=1$, it gives Wilf's inequality. When $r=2$, since $w_2(G)=2m$, it gives Nikiforov's spectral Tur\'an theorem.

In some situations, the estimates using global graph parameters could be too rough.  The clique number
$\omega(G)$ is a typical example: it records the size of a largest clique
 in $G$, but it does not describe the clique structure around a
particular vertex or edge. Thus a vertex or an edge may lie only in small
cliques even when $\omega(G)$ is large. Replacing such global parameters by
local ones can therefore lead to sharper inequalities. This local point of view has been adopted in several recent works; see, for example,
\cite{Bradac22,MalecTompkins23,KirschNir24,AdakChandran25,
AdakChandranEG25,AdakChandranGT25,AdakChandranZykov25,ELW24,KKP25,LN25,LN26,LiNing25}.

For a vertex $v\in V(G)$ and an edge $uv\in E(G)$, let $c_G(v)$ and $c_G(uv)$
denote the maximum order of a clique in $G$ containing $v$ and $uv$,
respectively. Also, let $w_r(v)$ denote the number of walks on $r$ vertices
starting at $v$.  Liu and Ning~\cite{LN25,LN26} strengthen Wilf's inequality  and Nikiforov's spectral Tur\'an theorem by using local clique parameters. They proved the
localized Wilf inequality
$\lambda(G)\le \sum_{v\in V(G)}(1-1/c_G(v))$, the edge-local
spectral Tur\'an inequality
$\lambda^2(G)\le 2\sum_{uv\in E(G)}(1-1/c_G(uv))$, and the
degree-local spectral Tur\'an inequality
$\lambda^2(G)\le \sum_{v\in V(G)}d_G(v)(1-1/c_G(v))$. These results
show that the cases $r=1$ and $r=2$ of Nikiforov's walk inequality admit local
strengthenings. It is therefore natural to ask whether the same is true for walks of every
length. Motivated by this question, Kannan, Kumar, and Pragada~\cite{KKP25} proved a
weaker local form of Nikiforov's walk inequality: for every connected graph
$G$ and every integer $r\ge 1$,
$\lambda^r(G)\le \sum_{v\in V(G)} w_r(v)\sqrt{(1-1/c_G(v))(1-1/\omega(G))}$. They further conjectured that the mixed local--global factor on the right-hand side of the inequality can be replaced by the fully local factor $1-1/c_G(v)$.

\begin{conjecture}[Kannan--Kumar--Pragada~\cite{KKP25}]\label{conj:kkp-main}
For every graph $G$ and every integer $r\ge 1$,
$$\lambda^r(G)
   \le
   \sum_{v\in V(G)}
   w_r(v)\frac{c_G(v)-1}{c_G(v)}.$$
\end{conjecture}
The cases $r=1$ and $r=2$ of Conjecture~\ref{conj:kkp-main} are already known. The case $r=1$ is the localized Wilf theorem of Liu and Ning~\cite{LN25}. The case $r=2$ follows from their degree-local spectral Tur\'an theorem~\cite{LN26}. However, the conjecture remains open for $r\ge 3$. In this paper, we prove Conjecture~\ref{conj:kkp-main}. In fact, we show a stronger edge-local inequality, where the local clique factor is attached to the first edge of the walk.

\begin{theorem}[Edge-local walk inequality]\label{thm:edge-main}
Let $G$ be a graph with at least one edge.  For every integer $r\ge 2$, we have
\begin{flalign}\label{eq:edge-local}
    \lambda^r(G)
   \le
   \sum_{uv\in E(G)}
   \frac{c_G(uv)-1}{c_G(uv)}
   \bigl(w_{r-1}(u)+w_{r-1}(v)\bigr).
\end{flalign}
Moreover, equality holds if and only if, after deleting all isolated vertices of $G$,
the remaining graph is either a regular complete $\omega(G)$-partite graph with
$\omega(G)\ge 3$, or a complete bipartite graph when $\omega(G)=2$; in the latter
case, the remaining complete bipartite graph is regular if $r$ is odd.
\end{theorem}
Theorem~\ref{thm:edge-main} immediately implies
Conjecture~\ref{conj:kkp-main} for all $r\ge 2$. Indeed, since
$c_G(uv)\le c_G(u)$ and $c_G(uv)\le c_G(v)$ for every edge $uv$, and since $(t-1)/t$ is increasing for $t\ge 1$, we have
\begin{flalign*}
&\sum_{uv\in E(G)}
   \frac{c_G(uv)-1}{c_G(uv)}
   \bigl(w_{r-1}(u)+w_{r-1}(v)\bigr)  \notag\\
&\le
\sum_{uv\in E(G)}
\left(
\frac{c_G(v)-1}{c_G(v)}w_{r-1}(u)
+
\frac{c_G(u)-1}{c_G(u)}w_{r-1}(v)
\right) \notag\\
&=
\sum_{v\in V(G)}
\frac{c_G(v)-1}{c_G(v)}
\sum_{u\in N_G(v)}w_{r-1}(u).
\end{flalign*}
For $r\ge 2$, we have
$w_r(v)=\sum_{u\in N_G(v)}w_{r-1}(u)$, since an $r$-vertex walk starting at
$v$ first moves to a neighbor $u$ and then continues as an $(r-1)$-vertex walk
starting at $u$. Therefore,
\begin{flalign}\label{eq:vertex-local-main}
\sum_{uv\in E(G)}
   \frac{c_G(uv)-1}{c_G(uv)}
   \bigl(w_{r-1}(u)+w_{r-1}(v)\bigr)
\le
\sum_{v\in V(G)}
w_r(v)\frac{c_G(v)-1}{c_G(v)}.
\end{flalign}
Combining Theorem~\ref{thm:edge-main} and \eqref{eq:vertex-local-main} proves Conjecture~\ref{conj:kkp-main} for every $r\ge 2$. The remaining case $r=1$ is exactly the localized Wilf theorem~\cite[Theorem~1.4]{LN25}, so we obtain the full vertex-local walk inequality.

\begin{corollary}[Vertex-local walk inequality]\label{cor:vertex-main}
For every graph $G$ and every integer $r\ge 1$, 
\begin{flalign}\label{eq:vertex-local}
   \lambda^r(G)
   \le
   \sum_{v\in V(G)}
   w_r(v)\frac{c_G(v)-1}{c_G(v)}. 
\end{flalign}
Moreover, equality holds if and only if either $G$ is edgeless, or, after deleting
all isolated vertices of $G$, the remaining graph is as follows: it is a regular
complete $\omega(G)$-partite graph when $r=1$ or $\omega(G)\ge 3$, and it is a
complete bipartite graph when $r\ge 2$ and $\omega(G)=2$; in the latter case,
the remaining complete bipartite graph is regular if $r$ is odd.
\end{corollary}

The vertex-local inequality in Corollary~\ref{cor:vertex-main} immediately implies Nikiforov's walk inequality in Theorem~\ref{thm:Nikiforov-walk}. Moreover, for $r=2$, Theorem~\ref{thm:edge-main} is exactly  Liu and Ning's edge-local spectral Tur\'an theorem. Thus our results may be viewed both as a local extension of Nikiforov's walk inequality and a walk-length extension of Liu and Ning's edge-local spectral Tur\'an theorem.

We briefly sketch the proof. Starting from a Perron vector of $A(G)$, we construct a Markov transition matrix and apply an entropy-type Perron-root estimate to this chain. This yields the key comparison lemma, which controls the contribution of long walks  in terms of the first edge. Combining this comparison with Liu and Ning's weighted edge-local spectral inequality gives Theorem~\ref{thm:edge-main}. The novelty is that the Perron-vector Markov chain provides a mechanism for carrying local edge information through walks of arbitrary length.

The remainder of the paper is organized as follows. In Section~\ref{sec:markov-chain}, we prove the Markov-chain estimate based on a Perron vector. In Section~\ref{sec:localized-walk-inequality}, we prove
Theorem~\ref{thm:edge-main} and then complete the proof of
Corollary~\ref{cor:vertex-main} by determining its extremal graphs. Finally,
Section~\ref{sec:concluding-remarks} records short comparisons among the edge-local, vertex-local, and global bounds.

\section{A Markov-chain estimate from a Perron vector}\label{sec:markov-chain} 
The main purpose of this section is to prove a Markov-chain estimate based on a
Perron vector. Before the proof, we fix some notation that will be used later.  Let $S=\{v_1,\ldots,v_n\}$ be a finite set. A matrix
$P=(p_{ij})_{n\times n}$ is called a Markov transition matrix on $S$ if
$p_{ij}\ge 0$ for all $1\le i,j\le n$ and
$\sum_{j=1}^{n}p_{ij}=1$ for every $1\le i\le n$. Here $p_{ij}$ is the
probability of moving from $v_i$ to $v_j$ in one step. A probability row vector $\pi=(\pi_1,\pi_2,\ldots,\pi_n)$ is stationary for
$P$ if $\pi_i\ge 0$ for all $1\le i\le n$, $\sum_{i=1}^n \pi_i=1$, and
$\pi P=\pi$. The all-one vector is denoted by
$\one$.

Let $B=(b_{ij})_{n\times n}$ be a nonnegative matrix. We say that $P$ is supported by $B$ if $p_{ij}>0$ always implies $b_{ij}>0$.  In logarithmic sums involving the transition probabilities $p_{ij}$, terms with $p_{ij}=0$ are omitted; equivalently, we use the convention $0\ln 0=0$.

The following lemma is a standard finite-dimensional entropy-type estimate for
the Perron root. It is closely related to the variational formulas of
Donsker--Varadhan~\cite{DonskerVaradhan75} and
Friedland--Karlin~\cite{FriedlandKarlin75}.
\begin{lemma}
\label{lem:irreducible}
Let $B=(b_{ij})_{n\times n}$ be a nonnegative irreducible matrix with spectral
radius $\lambda(B)$, and let $P=(p_{ij})_{n\times n}$ be a Markov transition
matrix with stationary distribution $\pi$. Assume that $P$ is supported by
$B$. Then
$$
\ln\lambda(B)
\ge
\sum_{\substack{i,j\\ p_{ij}>0}}
\pi_i p_{ij}\ln\frac{b_{ij}}{p_{ij}}.
$$
\end{lemma}
\begin{proof}[\bf Proof]
Let $y=(y_1,\ldots,y_n)^T$ be a Perron vector of $B$ corresponding to $\lambda(B)$. Then $By=\lambda(B)y$. In particular, for any fixed $i\in\{1,\ldots,n\}$, $\sum_{j=1}^{n} b_{ij}y_j=\lambda(B)y_i$. Since $y_i>0$, we have $\lambda(B)=\sum_{j=1}^{n}(b_{ij}y_j)/y_i$ and hence
$\ln\lambda(B)=\ln\sum_{j=1}^{n}(b_{ij}y_j)/y_i$. Since all terms are nonnegative,
\begin{flalign*}
   \ln\lambda(B)
   \ge
   \ln\sum_{j:p_{ij}>0}\frac{b_{ij}y_j}{y_i}
   =
   \ln\sum_{j:p_{ij}>0}p_{ij}\frac{b_{ij}y_j}{p_{ij}y_i}.
\end{flalign*}
When $p_{ij}>0$, the support assumption gives $b_{ij}>0$. Since also
$y_i,y_j>0$, the quantities $b_{ij}y_j/(p_{ij}y_i)$ are strictly positive.
Applying Jensen's inequality to the concave function $\ln t$, we get
\begin{flalign*}
   \ln\lambda(B)
   \ge
   \sum_{j:p_{ij}>0}p_{ij}
   \ln\frac{b_{ij}y_j}{p_{ij}y_i}.
\end{flalign*}
Multiplying by $\pi_i$ and summing over $i$, and using $\sum_i\pi_i=1$, we get
\[
   \ln\lambda(B)
   \ge
   \sum_{\substack{i,j\\ p_{ij}>0}}\pi_i p_{ij}
   \ln\frac{b_{ij}y_j}{p_{ij}y_i}.
\]
Hence
\begin{align*}
   \ln\lambda(B)
   \ge
   \sum_{\substack{i,j\\ p_{ij}>0}}\pi_i p_{ij}\ln\frac{b_{ij}}{p_{ij}}+
   \sum_{i,j}\pi_i p_{ij}\ln y_j
   -
   \sum_{i,j}\pi_i p_{ij}\ln y_i.
\end{align*}
Since $\pi$ is stationary for
$P$, we have
$
\sum_{i=1}^{n}\pi_i p_{ij}=\pi_j   
$ for every $j$.
Therefore
\[
   \sum_{i,j}\pi_i p_{ij}\ln y_j
   =
   \sum_{j=1}^{n}
   \left(\sum_{i=1}^{n}\pi_i p_{ij}\right)\ln y_j
   =
   \sum_{j=1}^{n}\pi_j\ln y_j.
\]
Note that $ \sum_{j=1}^{n}p_{ij}=1 $ for every $i$. This implies that 
\[
   \sum_{i,j}\pi_i p_{ij}\ln y_i
   =
   \sum_{i=1}^{n}\pi_i
   \left(\sum_{j=1}^{n}p_{ij}\right)\ln y_i
   =
   \sum_{i=1}^{n}\pi_i\ln y_i.
\]
Thus, $ \ln\lambda(B)
   \ge\sum_{\substack{i,j\\ p_{ij}>0}}\pi_i p_{ij}\ln\frac{b_{ij}}{p_{ij}}.$ This completes the proof of Lemma~\ref{lem:irreducible}.
\end{proof}

A weighted graph is a pair $W=(G,\mu)$, where $\mu:E(G)\to \mathbb{R}_{\ge 0}$
is an edge-weight function. The weighted adjacency matrix of $W$ is denoted by
$A(W)=(w_{uv})_{u,v\in V(G)}$, where $w_{uv}=\mu(uv)$ if $uv\in E(G)$, and
$w_{uv}=0$ otherwise. We write $\lambda(W)$ for the largest eigenvalue of
$A(W)$.  
\begin{lemma}\label{lem:main}
Let $G$ be a connected graph of order $n\ge 2$, and let $A(G)=(a_{ij})_{n\times n}$. For an
integer $k\ge 0$, set $z=A^k(G)\one$. Define a weighted graph
$W_k=(G,\mu_k)$ by
$
   \mu_k(uv)=\sqrt{(z_u+z_v)/2}
$
for every edge $uv\in E(G)$. Then
$
   \lambda(W_k)\ge \lambda^{1+\frac{k}{2}}(G).
$ Moreover, equality holds if and only if $A^k(G)\one=\lambda^k(G)\one$.
\end{lemma}
\begin{proof}[\bf Proof]
Since $G$ is connected, $A(G)$ is a nonnegative irreducible matrix and hence $\lambda(G)>0$. By the Perron--Frobenius theorem, there exists a
positive unit Perron vector $x=(x_1,\ldots,x_n)^T$ corresponding to
$\lambda(G)$. That is, $A(G)x=\lambda(G)x$ and $\sum_{i=1}^n x_i^2=1$. We now define a matrix $P=(p_{ij})_{n\times n}$ by
$$
p_{ij}=\frac{a_{ij}x_j}{\lambda(G)x_i}.
$$
We claim that $P$ is a Markov transition matrix. First, $p_{ij}\ge 0$ for all $i,j$, since $a_{ij}\ge 0$ and $x_i,x_j>0$. Next, for every $i\in \{1,\ldots,n\}$,
\begin{flalign*}
   \sum_{j=1}^{n} p_{ij}
   =
   \frac{1}{\lambda(G)x_i}\sum_{j=1}^{n} a_{ij}x_j 
   =
   \frac{(A(G)x)_i}{\lambda(G)x_i} 
   =
   \frac{\lambda(G)x_i}{\lambda(G)x_i} 
   =
   1.
\end{flalign*}
Therefore $P$ is indeed a Markov transition matrix. Next we show that the  vector $\pi=(\pi_1,\pi_2,\ldots,\pi_n)$ defined by
$\pi_i=x_i^2$
is stationary for $P$. Since $\sum_{i=1}^n x_i^2=1$,  $\sum_{i=1}^n \pi_i=1$. Moreover, for every $j\in \{1,\ldots,n\}$,
\begin{flalign*}
   \sum_{i=1}^{n}\pi_i p_{ij}
   =
   \sum_{i=1}^{n} x_i^2 \frac{a_{ij}x_j}{\lambda(G)x_i} 
   =
   \frac{x_j}{\lambda(G)}\sum_{i=1}^{n} a_{ij}x_i 
=
   \frac{x_j}{\lambda(G)}\lambda(G)x_j 
   =
   x_j^2 
   &=
   \pi_j.
\end{flalign*}
This implies that 
$$
\pi P=\left(\sum_{i=1}^{n}\pi_i p_{i1},\ldots,\sum_{i=1}^{n}\pi_i p_{in}\right)
=(\pi_1,\ldots,\pi_n)=\pi,
$$
so $\pi$ is stationary for $P$.

The vector $z=A^k(G)\one$ has strictly positive coordinates.  This is clear when
$k=0$, since then $z=\one$.  If $k\ge 1$, then every vertex of $G$ has positive
degree because $G$ is connected and has at least two vertices.  Hence from each
vertex there is a walk of length $k$, so every coordinate of $z$ is positive.
Thus $A(W_k)$ has the same support as $A(G)$, and therefore $A(W_k)$ is irreducible. Write $A(W_k)=(w_{ij})_{n\times n}$. Applying Lemma~\ref{lem:irreducible} with $B=A(W_k)$ and $P$, we have 
\begin{equation}\label{eq:2}
  \ln \lambda(W_k)
  \ge
  \sum_{\substack{i,j\\ p_{ij}>0}}\pi_i p_{ij}\ln\frac{w_{ij}}{p_{ij}}
  =
  \sum_{\substack{i,j\\ p_{ij}>0}}\pi_i p_{ij}\ln\frac{1}{p_{ij}}
  +
  \sum_{\substack{i,j\\ p_{ij}>0}}\pi_i p_{ij}\ln w_{ij}.
\end{equation}
Since the sums are taken over pairs with $p_{ij}>0$, we have $a_{ij}=1$ for every term. Hence $1/p_{ij}=\lambda(G)x_i/x_j$ in these sums. Therefore 
\begin{flalign}\label{eq:3}
    \sum_{\substack{i,j\\ p_{ij}>0}}\pi_i p_{ij}\ln\frac{1}{p_{ij}}
    = \sum_{\substack{i,j\\ p_{ij}>0}}\pi_i p_{ij}\ln\frac{\lambda(G)x_i}{x_j}
    \notag=& \ln \lambda(G)\sum_{i,j}\pi_i p_{ij}+ \sum_{i,j}\pi_i p_{ij}\ln x_i - \sum_{i,j}\pi_i p_{ij}\ln x_j\\
    =& \ln \lambda(G)\sum_{i=1}^n\pi _i=\ln \lambda(G).
\end{flalign}
By \eqref{eq:2} and \eqref{eq:3}, we  have $\ln \lambda(W_k)\ge \ln \lambda(G)+  \sum_{\substack{i,j\\ p_{ij}>0}}\pi_i p_{ij}\ln w_{ij}$. If $v_iv_j\in E(G)$, then
$\ln w_{ij}=\frac{1}{2}\ln{\frac{z_i+z_j}{2}}$.  Since
$\frac{z_i+z_j}{2}\ge (z_iz_j)^{1/2}$, 
\begin{flalign}
   \ln w_{ij}= \frac{1}{2}\ln{\frac{z_i+z_j}{2}}\ge \frac{1}{4}(\ln z_i+\ln z_j).
\end{flalign}
Therefore,
\begin{flalign*}
    \sum_{\substack{i,j\\ p_{ij}>0}}\pi_i p_{ij}\ln w_{ij}
    \ge  \frac{1}{4}\sum_{i,j}\pi_i p_{ij}(\ln z_i+\ln z_j)
    = \frac{1}{4}\sum_{i=1}^n\pi_i \ln z_i
      + \frac{1}{4}\sum_{j=1}^n\pi_j \ln z_j 
    =\frac{1}{2}\sum_{i=1}^n\pi_i \ln z_i.
\end{flalign*}
Next, it suffices to prove 
\begin{flalign}\label{main-eq}
   \sum_{i=1}^n\pi_i \ln z_i\ge k\ln \lambda(G).
\end{flalign}
By the definition of $P$, we have  $P=\lambda(G)^{-1}D_x^{-1}A(G)D_x$, where 
\begin{flalign*}
  D_x=
\begin{pmatrix}
x_1 & 0   & \cdots & 0 \\
0   & x_2 & \cdots & 0 \\
\vdots & \vdots & \ddots & \vdots \\
0   & 0   & \cdots & x_n
\end{pmatrix}.
\end{flalign*}
Hence, for every $k\ge 0$, $P^k=\frac{1}{\lambda^k(G)}D_x^{-1}A^k(G)D_x$. We write $P^k=(p^{(k)}_{ij})_{n\times n}$ and $A^k(G)=(a^{(k)}_{ij})_{n\times n}$. Then $p^{(k)}_{ij}=\frac{a^{(k)}_{ij}x_j}{\lambda^k(G)x_i}.$
This implies that  $$z_i=\sum_{j=1}^n a^{(k)}_{ij}=\sum_{j=1}^n \frac{\lambda^k(G)p^{(k)}_{ij}x_i}{x_j}=\lambda^k(G)x_i\sum_{j=1}^n p^{(k)}_{ij}\frac{1}{x_j}.$$
Therefore
$$\ln z_i=\ln\left(\lambda^k(G)x_i\sum_{j=1}^n p^{(k)}_{ij}\frac{1}{x_j}\right)=\ln{\lambda^k(G)}+\ln x_i+\ln\left(\sum_{j=1}^n p^{(k)}_{ij}\frac{1}{x_j}\right).$$

Note that $P\one=\one$. Hence $P^k\one=\one$ and $\sum_{j=1}^n p^{(k)}_{ij}=1$.
Since  $\ln t$ is concave, Jensen's inequality yields
\begin{flalign*}
\ln\left(\sum_{j=1}^n p^{(k)}_{ij}\frac{1}{x_j}\right)
  \ge
\sum_{j=1}^n p^{(k)}_{ij}\ln\frac{1}{x_j}.
\end{flalign*}
Hence
\begin{equation}
\begin{aligned}
  \ln z_i
  \ge k\ln\lambda(G)+
  \ln x_i+\sum_{j=1}^n p^{(k)}_{ij}\ln\frac{1}{x_j} 
  = k\ln\lambda(G)+
  \ln x_i-\sum_{j=1}^n p^{(k)}_{ij}\ln{x_j}.
\end{aligned}
\label{eq:Jensen-last}
\end{equation}
Multiplying \eqref{eq:Jensen-last} by $\pi_i$ and summing over $1\le i\le n$, we have
\begin{flalign*}
    \sum_{i=1}^n \pi_i\ln z_i \ge \sum_{i=1}^n k\ln \lambda(G)\pi_i+ \sum_{i=1}^n \pi_i\ln x_i- \sum_{i=1}^n\pi_i\sum_{j=1}^n p^{(k)}_{ij}\ln{x_j}
\end{flalign*}
Note that $\pi P=\pi$. Therefore, $\pi P^k=\pi$. 
Then
\begin{align}\label{eq:finally}
    \notag\sum_{i=1}^n \pi_i\ln z_i
    &\ge k\ln \lambda(G)
      +\sum_{i=1}^n \pi_i\ln x_i
      -\sum_{j=1}^n\ln{x_j}\sum_{i=1}^n \pi_i p^{(k)}_{ij}\\
    &= k\ln \lambda(G)
      +\sum_{i=1}^n \pi_i\ln x_i
      \notag  -\sum_{j=1}^n\ln{x_j}\pi_j\\
  &= k\ln \lambda(G).
\end{align}
This proves \eqref{main-eq}, and hence  $\lambda(W_k)\ge \lambda^{1+\frac{k}{2}}(G)$.

Finally, we prove the equality statement. Suppose first that
$\lambda(W_k)=\lambda^{1+\frac{k}{2}}(G)$. Then equality holds throughout the
chain of inequalities above. In particular, equality holds in
$(z_i+z_j)/2\ge \sqrt{z_i z_j}$ for every edge $v_iv_j\in E(G)$. Since $z$ has positive entries, this gives $z_i=z_j$ for every edge
$v_iv_j\in E(G)$. As $G$ is connected, we have $z=c\one$ for some constant
$c$. Using the positive Perron vector $x$ of $A(G)$, we obtain
$$
c x^T\one
=x^Tz
=x^T A^k(G)\one
=(A^k(G)x)^T\one
=\lambda^k(G)x^T\one.
$$
Since $x^T\one>0$, it follows that $c=\lambda^k(G)$. Hence
$A^k(G)\one=\lambda^k(G)\one$.

Conversely, suppose that $A^k(G)\one=\lambda^k(G)\one$. Then
$z=\lambda^k(G)\one$, and so $\mu_k(uv)=\lambda^{k/2}(G)$ for every edge
$uv\in E(G)$. Hence $A(W_k)=\lambda^{k/2}(G)A(G)$, and therefore
$$
\lambda(W_k)=\lambda^{k/2}(G)\lambda(G)=\lambda^{1+\frac{k}{2}}(G).
$$
This completes the proof of Lemma~\ref{lem:main}.
\end{proof}
Recall that a
connected bipartite graph is called semiregular if, with respect to its
bipartition $X\cup Y$, all vertices in $X$ have the same degree and all
vertices in $Y$ have the same degree.

\begin{lemma}\label{lem:main-equality-1}
Let $G$ be a connected graph of order $n\ge 2$, and let $k\ge 1$ be an integer.
Then $A^k(G)\one=\lambda^k(G)\one$ holds if and only if $G$ is regular when
$k$ is odd, and $G$ is either regular or semiregular bipartite when $k$ is
even.
\end{lemma}
\begin{proof}[\bf Proof]
Since $G$ is connected, $A(G)$ is a real symmetric nonnegative irreducible
matrix. Hence, we may choose an orthonormal eigen-basis
$\alpha_1,\alpha_2,\ldots,\alpha_n$ of $\mathbb R^n$ with corresponding
eigenvalues $\lambda_1,\lambda_2,\ldots,\lambda_n$, ordered so that
$\lambda_1=\lambda(G)$. Thus $A(G)\alpha_j=\lambda_j\alpha_j$ for every
$1\le j\le n$, and $\alpha_1$ may be chosen as the normalized positive Perron
vector. Write $\one=\sum_{j=1}^n \gamma_j\alpha_j$. Then
\begin{flalign*}
  A^k(G)\one= A^k(G)\left(\sum_{j=1}^n \gamma_j\alpha_j\right)=\sum_{j=1}^n\gamma_j\lambda_j^k\alpha_j,  
\end{flalign*}
whereas $\lambda^k(G)\one=\sum_{j=1}^n\gamma_j\lambda^k(G)\alpha_j$. Therefore, $A^k(G)\one=\lambda^k(G)\one$ is equivalent to $$\sum_{j=1}^n\gamma_j(\lambda_j^k-\lambda_1^k)\alpha_j=0.$$
Since $\alpha_1,\ldots,\alpha_n$ are linearly independent, this implies that  $ \gamma_j(\lambda_j^k-\lambda_1^k)=0 $ for every $j$. Thus every eigenvector that appears with nonzero coefficient in the expansion
of $\one$ has eigenvalue $\lambda_j$ satisfying $\lambda_j^k=\lambda_1^k$. If $k$ is odd, then every eigenvector that appears with nonzero coefficient in the expansion
of $\one$ has eigenvalue $\lambda_j$ satisfying $\lambda_j=\lambda_1$. It follows that 
$$A(G)\one=A(G)\sum_{i:\gamma_i\neq 0}\gamma_i\alpha_i=\sum_{i:\gamma_i\neq 0}\gamma_iA(G)\alpha_i=\lambda_1\sum_{i:\gamma_i\neq 0}\gamma_i\alpha_i=\lambda_1\one.$$
Therefore, $G$ is regular. Now assume that $k$ is even.   Then an eigenvalue appearing in the expansion of
$\one$ can only be $\lambda_1$ or $-\lambda_1$. If $G$ is not bipartite, then
$-\lambda_1$ is not an eigenvalue by the standard spectral characterization of
connected bipartite graphs \cite[Theorem~3.11]{BrouwerHaemers12}.   Thus every eigenvector appearing with nonzero coefficient in the expansion of
$\one$ has eigenvalue $\lambda_1$. Hence $A(G)\one=\lambda_1\one$, and so $G$ is regular.  

Assume now that $G$ is bipartite, and let $X\cup Y$ be its bipartition. Since
$G$ is connected, the Perron root $\lambda_1$ is simple, and hence the $\lambda_1$-eigenspace is spanned by $\alpha_1$. Since
$G$ is bipartite, the vector $\alpha_1'$ obtained from $\alpha_1$ by changing
the sign on one part, that is,
$$
(\alpha_1')_v=
\begin{cases}
(\alpha_1)_v, & v\in X,\\
-(\alpha_1)_v, & v\in Y,
\end{cases}
$$
is an eigenvector corresponding to $-\lambda_1$. Moreover, the
$-\lambda_1$-eigenspace is spanned by $\alpha_1'$.

Thus the condition $A^k(G)\one=\lambda^k(G)\one$ implies that
$\one$ lies in the span of $\alpha_1$ and $\alpha_1'$. Hence
$\one=a\alpha_1+b\alpha_1'$ for some real numbers $a,b$. Therefore, $(\alpha_1)_v=1/(a+b)$ for every $v\in X$, and
$(\alpha_1)_v=1/(a-b)$ for every $v\in Y$. It follows that $\alpha_1$ is constant on $X$ and constant on $Y$. Write $(\alpha_1)_v=s$ for $v\in X$ and $(\alpha_1)_v=t$ for $v\in Y$, where
$s,t>0$. From $A(G)\alpha_1=\lambda(G)\alpha_1$, for every $u\in X$ we have
$d_G(u)t=\lambda(G)s$, and for every $v\in Y$ we have
$d_G(v)s=\lambda(G)t$. Hence all vertices in $X$ have the same degree and all
vertices in $Y$ have the same degree. Thus $G$ is semiregular bipartite.

Conversely, if $G$ is regular, say $d$-regular, then
$A(G)\one=d\one$ and $\lambda(G)=d$. Hence
$A^k(G)\one=\lambda^k(G)\one$ for every $k\ge 1$.

Now suppose that $G$ is semiregular bipartite with bipartition $X\cup Y$. Let
$d_X$ and $d_Y$ be the common degrees of the vertices in $X$ and $Y$,
respectively. Then $A^2(G)\one=d_Xd_Y\one$: indeed, for $u\in X$,
$(A^2(G)\one)_u$ is the sum of the degrees of the neighbors of $u$, which is
$d_Xd_Y$, and the same argument applies to vertices in $Y$. Therefore
$\lambda^2(G)=d_Xd_Y$. Hence, for every even $k$,
\begin{align*}
A^k(G)\one
=(A^2(G))^{k/2}\one 
=(d_Xd_Y)^{k/2}\one 
=\lambda^k(G)\one.
\end{align*}
Thus semiregular bipartite graphs satisfy
$A^k(G)\one=\lambda^k(G)\one$ for every even $k$. This completes the proof of
Lemma~\ref{lem:main-equality-1}.
\end{proof}
\section{The localized walk inequality}\label{sec:localized-walk-inequality}

In this section, we prove the edge-local and vertex-local walk inequalities and
determine their extremal graphs. We shall use the standard walk-counting
identities $w_r(v)=(A^{r-1}(G)\one)_v$ and
$w_r(G)=\one^T A^{r-1}(G)\one$. We first recall three results of Liu and Ning that will be used below.

\begin{theorem}[Liu--Ning~\cite{LN25}]
\label{thm:liu-ning-weighted}
Let $W=(G,\mu)$ be a weighted graph. Then
$$\lambda^2(W)
   \le
   2\sum_{uv\in E(G)}
   \frac{c_G(uv)-1}{c_G(uv)}\mu^2(uv).$$
\end{theorem}

\begin{theorem}[Liu--Ning~\cite{LN25}]
\label{thm:r-one}
For every graph $G$,
$$
\lambda(G)\le \sum_{v\in V(G)}\frac{c_G(v)-1}{c_G(v)}.
$$
Equality holds if and only if $G$ is edgeless or, after deleting isolated
vertices, the remaining graph is a regular complete multipartite graph.
\end{theorem}

\begin{theorem}[Liu--Ning~\cite{LN26}]
\label{thm:r-two}
Let $G$ be a graph with at least one edge. Then
$$
\lambda^2(G)\le 2\sum_{uv\in E(G)}\frac{c_G(uv)-1}{c_G(uv)}.
$$
Equality holds if and only if, after deleting isolated vertices, the remaining
graph is either a complete bipartite graph, or a regular complete
$\omega(G)$-partite graph with $\omega(G)\ge 3$.
\end{theorem}
\begin{proof}[\bf Proof of Theorem~\ref{thm:edge-main}]
We first assume that $G$ is connected. Set $z=A^{r-2}(G)\one$. Then
$z_v=w_{r-1}(v)$ for every $v\in V(G)$. Define a weighted graph
$W_{r-2}=(G,\mu_{r-2})$ by
$\mu_{r-2}(uv)=\sqrt{(z_u+z_v)/2}$ for every edge $uv\in E(G)$. By Lemma~\ref{lem:main} with $k=r-2$, we have
$ \lambda^{r/2}(G)\le\lambda(W_{r-2})$. Hence
$\lambda^r(G)\le \lambda^2(W_{r-2})$.  By Theorem~\ref{thm:liu-ning-weighted}, we have that 
\begin{flalign*}
  \lambda^r(G) \le \lambda^2(W_{r-2}) &\le
   2\sum_{uv\in E(G)}
   \frac{c_G(uv)-1}{c_G(uv)}\mu_{r-2}^2(uv)\\&=\sum_{uv\in E(G)}
   \frac{c_G(uv)-1}{c_G(uv)}(z_u+z_v)\\&=\sum_{uv\in E(G)}
   \frac{c_G(uv)-1}{c_G(uv)}(w_{r-1}(u)+w_{r-1}(v)).
\end{flalign*}
This proves Theorem~\ref{thm:edge-main} for connected graphs. Now assume that $G$ is not connected. Let $G_1$ be a connected component of $G$ such that
$\lambda(G)=\lambda(G_1)$. For edges in $G_1$, the clique numbers and walk counts agree with those in $G$. Hence
\begin{flalign*}
   \lambda^r(G)=\lambda^r(G_1)\le& \sum_{uv\in E(G_1)}
   \frac{c_G(uv)-1}{c_G(uv)}(w_{r-1}(u)+w_{r-1}(v))\\\le& \sum_{uv\in E(G)}
   \frac{c_G(uv)-1}{c_G(uv)}(w_{r-1}(u)+w_{r-1}(v)),
\end{flalign*}
completing the proof of \eqref{eq:edge-local}.

Now we characterize all the graphs attaining equality in \eqref{eq:edge-local}. The case $r=2$ is exactly the equality case of Theorem~\ref{thm:r-two}. Hence assume that $r\ge3$. By the preceding proof, equality in \eqref{eq:edge-local} implies that $G$ has exactly one nontrivial connected component, say $G'$. Since isolated vertices do not affect either side of \eqref{eq:edge-local}, we may work on $G'$.

Set $z=A^{r-2}(G')\one$ and define $W_{r-2}$ on $G'$ as in the connected case. Equality in the first step of the connected-case proof gives $\lambda(W_{r-2})=\lambda^{r/2}(G')$. By the equality statement in Lemma~\ref{lem:main} with $k=r-2$, we have
$A^{r-2}(G')\one=\lambda^{r-2}(G')\one$. Hence $z=\lambda^{r-2}(G')\one$. It follows that
\begin{flalign*}
\lambda^r(G')=
2\lambda^{r-2}(G')\sum_{uv\in E(G')}
   \frac{c_{G'}(uv)-1}{c_{G'}(uv)}.
\end{flalign*}
Since $G'$ has at least one edge, $\lambda(G')>0$. Therefore,
\begin{flalign*}
\lambda^2(G')=2\sum_{uv\in E(G')}
   \frac{c_{G'}(uv)-1}{c_{G'}(uv)}.
\end{flalign*}
By the equality case of Theorem~\ref{thm:r-two}, $G'$ is either a complete
bipartite graph or a regular complete $\omega(G')$-partite graph with
$\omega(G')\ge 3$. Applying Lemma~\ref{lem:main-equality-1} with $k=r-2$ shows that
$G'$ is regular when $r$ is odd. Hence the bipartite case is regular when $r$
is odd. Since $G'$ is the unique nontrivial component, $\omega(G')=\omega(G)$.

Conversely, suppose that $G$ is one of the listed graphs. Let $L_r(G)$ and
$V_r(G)$ denote the right-hand sides of \eqref{eq:edge-local} and
\eqref{eq:vertex-local}, respectively. By the inequality just proved,
\eqref{eq:vertex-local-main}, and $c_G(v)\le \omega(G)$ for every vertex $v$,
\begin{flalign}\label{eq:long-chain}
\lambda^r(G)\le L_r(G)\le V_r(G)\le
w_r(G)\frac{\omega(G)-1}{\omega(G)}.
\end{flalign}
Since isolated vertices do not affect either side of \eqref{eq:edge-local}, we
may assume that $G$ is connected. For the listed graphs,
Theorem~\ref{thm:Nikiforov-walk} gives equality in the last term of
\eqref{eq:long-chain}. Hence every inequality in \eqref{eq:long-chain} is an
equality, and in particular equality holds in \eqref{eq:edge-local}.
This completes the proof of Theorem~\ref{thm:edge-main}.
\end{proof}
We next give a short proof of Corollary~\ref{cor:vertex-main}. For $r=1$, the
inequality is Theorem~\ref{thm:r-one}. For $r\ge2$, it follows from
Theorem~\ref{thm:edge-main} and \eqref{eq:vertex-local-main}. Thus it remains
to determine the graphs attaining equality in \eqref{eq:vertex-local}.
\begin{proof}[\bf Proof of Corollary~\ref{cor:vertex-main}]
The inequality follows from Theorem~\ref{thm:r-one} when $r=1$, and from
Theorem~\ref{thm:edge-main} together with \eqref{eq:vertex-local-main} when
$r\ge2$. We now determine the graphs attaining equality in \eqref{eq:vertex-local}. If $G$ is edgeless, then equality
is clear.

If $r=1$, the remaining equality cases are exactly those of the localized Wilf
theorem, Theorem~\ref{thm:r-one}. Hence assume
that $r\ge 2$ and that $G$ has at least one edge. Let $L_r(G)$ and $V_r(G)$
denote the right-hand sides of \eqref{eq:edge-local} and \eqref{eq:vertex-local},
respectively. By Theorem~\ref{thm:edge-main} and \eqref{eq:vertex-local-main},
$\lambda^r(G)\le L_r(G)\le V_r(G)$. Thus equality in \eqref{eq:vertex-local}
implies equality in \eqref{eq:edge-local}. Hence Theorem~\ref{thm:edge-main}
gives precisely the $r\ge2$ non-edgeless graphs listed in the statement.

Conversely, suppose first that $r=1$ and that $G$ satisfies the corresponding
non-edgeless condition in the statement. Then equality follows from
Theorem~\ref{thm:r-one}. Finally, suppose that $r\ge2$ and that $G$ satisfies
the corresponding structural condition in the statement. By \eqref{eq:long-chain}
and the equality case of Theorem~\ref{thm:Nikiforov-walk} for these graphs, we
have $\lambda^r(G)=V_r(G)$. Thus equality holds in \eqref{eq:vertex-local}. This
completes the proof of Corollary~\ref{cor:vertex-main}.
\end{proof}

\section{A comparison of the three bounds}\label{sec:concluding-remarks}

For \(r\ge 2\), define
\[
\begin{cases}
\displaystyle
E_r(G)=
\sum_{uv\in E(G)}
   \frac{c_G(uv)-1}{c_G(uv)}
   \bigl(w_{r-1}(u)+w_{r-1}(v)\bigr),\\[2ex]
\displaystyle
V_r(G)=
\sum_{v\in V(G)}
   w_r(v)\frac{c_G(v)-1}{c_G(v)},\\[2ex]
\displaystyle
U_r(G)=
   w_r(G)\frac{\omega(G)-1}{\omega(G)}.
\end{cases}
\]
Then Theorem~\ref{thm:edge-main}, \eqref{eq:vertex-local-main}, and
$c_G(v)\le \omega(G)$ give
\[
\lambda^r(G)\le E_r(G)\le V_r(G)\le U_r(G).
\]
The next two examples show that both comparisons can be strict.

\begin{example}[Edge-local versus vertex-local]\label{ex:edge-vs-vertex}
Let $s\ge 3$, and let $H$ be a triangle-free graph with a distinguished vertex
$x$ such that $d_H(x)>0$. Let $G$ be obtained by identifying $x$ with one
vertex of $K_s$, and by adding no other edges between $K_s\setminus\{x\}$ and
$H\setminus\{x\}$. Then $c_G(x)=s$, while $c_G(xu)=2$ for every
$u\in N_H(x)$. For all other edges, the edge coefficient and the corresponding
vertex coefficients coincide. Hence, for every $r\ge 2$,
\[
\begin{aligned}
V_r(G)-E_r(G)
&=\left(\frac{s-1}{s}-\frac12\right)
  \sum_{u\in N_H(x)}w_{r-1}(u)  \\
&=\frac{s-2}{2s}\sum_{u\in N_H(x)}w_{r-1}(u)>0.
\end{aligned}
\]
Thus the edge-local bound is strictly sharper than the vertex-local bound on
this class of graphs. The reason is that the vertex $x$ lies in an $s$-clique,
but each edge $xu$ with $u\in N_H(x)$ lies only in a $2$-clique. See
Figure~\ref{fig:edge-vs-vertex} for an illustration.
\end{example}

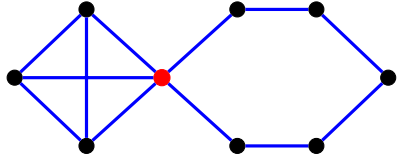
\begin{figure}[H]
\centering
\begin{tikzpicture}[scale=0.95]
  \node[gluevtx] (x) at (0,0) {};
  \node[vtx] (a) at (-1.05,0.95) {};
  \node[vtx] (b) at (-1.05,-0.95) {};
  \node[vtx] (c) at (-2.05,0) {};
  \draw[ed] (x)--(a) (x)--(b) (x)--(c) (a)--(b) (a)--(c) (b)--(c);

  \node[vtx] (u1) at (1.05,0.95) {};
  \node[vtx] (u2) at (1.05,-0.95) {};
  \node[vtx] (u3) at (2.15,0.95) {};
  \node[vtx] (u4) at (2.15,-0.95) {};
  \node[vtx] (u5) at (3.15,0) {};
  \draw[ed] (x)--(u1) (x)--(u2) (u1)--(u3) (u3)--(u5) (u5)--(u4) (u4)--(u2);
\end{tikzpicture}
\caption{A graph of the type in Example~\ref{ex:edge-vs-vertex}. The glued vertex is shown in red.}
\label{fig:edge-vs-vertex}
\end{figure}

\begin{example}[Vertex-local versus global]\label{ex:vertex-vs-global}
Let $G$ be the disjoint union of $K_s$ and a triangle-free graph $H$, where
$s\ge 3$ and $H$ has at least one edge. Then $\omega(G)=s$. For every
$r\ge 2$,
\[
   V_r(G)=
   \left(1-\frac1s\right)w_r(K_s)+\frac12 w_r(H),
\]
where isolated vertices of $H$, if any, do not contribute. On the other hand,
\[
   U_r(G)=
   \left(1-\frac1s\right)\bigl(w_r(K_s)+w_r(H)\bigr).
\]
Hence
\[
   U_r(G)-V_r(G)
   =\left(\frac12-\frac1s\right)w_r(H)>0.
\]
Thus the vertex-local bound is strictly sharper than the global bound on this
class of graphs. Here the clique $K_s$ determines the global parameter
$\omega(G)$, while walks inside the triangle-free component are charged with
the local coefficient $1/2$. See Figure~\ref{fig:vertex-vs-global} for an
illustration.
\end{example}

\begin{figure}[H]
\centering
\begin{tikzpicture}[scale=0.95]
  \node[vtx] (a1) at (-2.35,0.95) {};
  \node[vtx] (a2) at (-2.35,-0.95) {};
  \node[vtx] (a3) at (-3.35,0) {};
  \node[vtx] (a4) at (-1.35,0) {};
  \draw[ed] (a1)--(a2) (a1)--(a3) (a1)--(a4) (a2)--(a3) (a2)--(a4) (a3)--(a4);

  \node[vtx] (b1) at (0.75,0) {};
  \node[vtx] (b2) at (1.35,0.95) {};
  \node[vtx] (b3) at (2.45,0.95) {};
  \node[vtx] (b4) at (3.05,0) {};
  \node[vtx] (b5) at (2.45,-0.95) {};
  \node[vtx] (b6) at (1.35,-0.95) {};
  \draw[ed] (b1)--(b2) (b2)--(b3) (b3)--(b4) (b4)--(b5) (b5)--(b6) (b6)--(b1);
\end{tikzpicture}
\caption{A graph of the type in Example~\ref{ex:vertex-vs-global}.}
\label{fig:vertex-vs-global}
\end{figure}
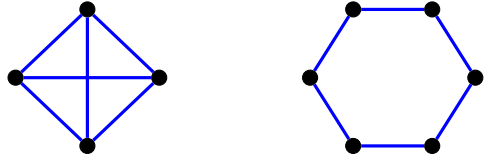

\section*{Acknowledgements}

This work was supported by the National Key R\&D Program of China
(No.~2022YFA1006400) and the National Natural Science Foundation of China
(No.~12571376).

\section*{Declaration}
\noindent$\textbf{Conflict~of~interest}$
The authors declare that they have no known competing financial interests or personal relationships that could have appeared to influence the work reported in this paper.
	
\noindent$\textbf{Data~availability}$
Data sharing not applicable to this paper as no datasets were generated or analysed during the current study.

\end{document}